\documentclass[12pt]{article}
\usepackage{amsmath,enumerate,amsfonts,amssymb,color,graphicx,amsthm,tikz}

\usepackage[colorlinks=true, allcolors=black]{hyperref}
\usepackage{todonotes}
\usepackage[normalem]{ulem}
\numberwithin{equation}{section}
\usepackage{cite}
\usepackage{mathtools}
\usepackage[nottoc,notlot,notlof]{tocbibind}
\usepackage[toc,page]{appendix}

\usepackage{lipsum}

\newlength{\leftstackrelawd}
\newlength{\leftstackrelbwd}
\def\leftstackrel#1#2{\settowidth{\leftstackrelawd}%
	{${{}^{#1}}$}\settowidth{\leftstackrelbwd}{$#2$}%
	\addtolength{\leftstackrelawd}{-\leftstackrelbwd}%
	\leavevmode\ifthenelse{\lengthtest{\leftstackrelawd>0pt}}%
	{\kern-.5\leftstackrelawd}{}\mathrel{\mathop{#2}\limits^{#1}}}

\theoremstyle{plain}
\newtheorem{thm}{Theorem}[section]
\newtheorem{lem}[thm]{Lemma}
\newtheorem{cor}[thm]{Corollary}
\newtheorem{prop}[thm]{Proposition}
\newtheorem*{thm*}{Theorem}

\theoremstyle{definition}

\newtheorem{rmk}[thm]{Remark}

\newtheorem{?}[thm]{Problem}

\newtheorem{innercustomthm}{Theorem}
\newenvironment{customthm}[1]
{\renewcommand\theinnercustomthm{#1}\innercustomthm}
{\endinnercustomthm}

\newtheorem{innercustomconj}{Conjecture}
\newenvironment{customconj}[1]
{\renewcommand\theinnercustomconj{#1}\innercustomconj}
{\endinnercustomconj}

\newcommand{\sg}{\sigma_1}

\newcommand{\ep}{\varepsilon}
\renewcommand{\epsilon}{\varepsilon}
\makeatletter
\def\@cite#1#2{[\textbf{#1\if@tempswa , #2\fi}]}
\def\@biblabel#1{[\textbf{#1}]}
\makeatother

\def\XXint#1#2#3{{\setbox0=\hbox{$#1{#2#3}{\int}$}
		\vcenter{\hbox{$#2#3$}}\kern-.5\wd0}}

\makeatletter
\newcommand*{\defeq}{\mathrel{\rlap{%
			\raisebox{0.3ex}{$\m@th\cdot$}}%
		\raisebox{-0.3ex}{$\m@th\cdot$}}%
	=}
\makeatother

\makeatletter
\newcommand*{\eqdef}{=\mathrel{\rlap{%
			\raisebox{0.3ex}{$\m@th\cdot$}}%
		\raisebox{-0.3ex}{$\m@th\cdot$}}%
	}
\makeatother

\newcounter{marnote}

\makeatletter
\def\underbracex#1#2{\mathop{\vtop{\m@th\ialign{##\crcr
				$\hfil\displaystyle{#2}\hfil$\crcr
				\noalign{\kern3\p@\nointerlineskip}%
				#1\crcr\noalign{\kern3\p@}}}}\limits}

\def\upbracefilla{$\m@th \setbox\z@\hbox{$\braceld$}%
	\bracelu\leaders\vrule \@height\ht\z@ \@depth\z@\hfill 
	\kern\p@\vrule \@width\p@\kern\p@\vrule \@width\p@\kern\p@\vrule \@width\p@
	$}

\def\upbracefillb{$\m@th \setbox\z@\hbox{$\braceld$}%
	\vrule \@width\p@\kern\p@\vrule \@width\p@\kern\p@\vrule \@width\p@\kern\p@
	\leaders\vrule \@height\ht\z@ \@depth\z@\hfill\bracerd
	\braceld\leaders\vrule \@height\ht\z@ \@depth\z@\hfill
	\kern\p@\vrule \@width\p@\kern\p@\vrule \@width\p@\kern\p@\vrule \@width\p@
	$}

\def\upbracefillc{$\m@th \setbox\z@\hbox{$\braceld$}%
	\vrule \@width\p@\kern\p@\vrule \@width\p@\kern\p@\vrule \@width\p@\kern\p@
	\leaders\vrule \@height\ht\z@ \@depth\z@\hfill
	\kern\p@\vrule \@width\p@\kern\p@\vrule \@width\p@\kern\p@\vrule \@width\p@
	$}

\def\upbracefilld{$\m@th \setbox\z@\hbox{$\braceld$}%
	\vrule \@width\p@\kern\p@\vrule \@width\p@\kern\p@\vrule \@width\p@\kern\p@
	\leaders\vrule \@height\ht\z@ \@depth\z@\hfill\braceru$}

\def\upbracefillbd{$\m@th \setbox\z@\hbox{$\braceld$}%
	\vrule \@width\p@\kern\p@\vrule \@width\p@\kern\p@\vrule \@width\p@\kern\p@
	\bracerd\braceld
	\leaders\vrule \@height\ht\z@ \@depth\z@\hfill\braceru$}

\makeatother

\begin{document}
	
	\title{\vspace*{-10mm}The first Steklov eigenvalue on manifolds with non-negative Ricci curvature and convex boundary}
	
	\author{Jonah A. J. Duncan and Aditya Kumar}
	\maketitle
	
	\vspace*{-10mm}\begin{abstract}
		We establish a new lower bound for the first non-zero Steklov eigenvalue of a \linebreak compact Riemannian manifold with non-negative Ricci curvature and (strictly) convex boundary. Related results are also obtained under weaker geometric hypotheses. 
	\end{abstract}

	\section{Introduction} 
	
	Let $(M^{n+1},g)$ be a smooth compact $(n+1)$-dimensional Riemannian manifold with non-empty boundary $\Sigma^n$ ($n\geq 1$). The Steklov eigenvalue problem on $M$ is to find $\sigma\in\mathbb{R}$ and $0\not\equiv f\in C^\infty(M)$ such that
	\begin{align}\label{16}
	\begin{cases}
	\Delta f = 0 & \text{in }M \\
	\frac{\partial f}{\partial \nu} = \sigma f & \text{on }\Sigma,
	\end{cases}
	\end{align}
	where $\Delta$ is the Laplace-Beltrami operator on $(M,g)$ and $\nu$ is the outward pointing unit normal to $\Sigma$. The set of $\sigma$ for which \eqref{16} admits a non-zero solution is called the Steklov spectrum of $M$, and its elements are the Steklov eigenvalues. It is well-known that the Steklov spectrum is non-negative, discrete and unbounded: 
	\begin{align*}
	0=\sigma_0<\sigma_1 \leq \sigma_2 \leq \dots \rightarrow+\infty.
	\end{align*}
	For additional background on the Steklov eigenvalue problem, we refer to the recent survey of Colbois et al.~\cite{CGGS24}. 

\subsection{Lower bounds on the first eigenvalue and Escobar's conjecture}
An important problem in geometric analysis is to establish a relationship between the \linebreak geometric invariants of a Riemannian manifold $(M,g)$ and lower bounds on the eigenvalues of (pseudo-)differential operators defined on $(M,g)$ and/or its boundary. For example, in the case that $M$ is closed and connected, Li and Yau \cite{LY80} established a lower bound for the first non-zero eigenvalue of the Laplace-Beltrami operator on $(M,g)$ in terms of bounds on the Ricci curvature and diameter. In the same paper they also considered the case that $M$ is compact with non-empty boundary $\Sigma$, and obtained a similar lower bound for the first non-zero Neumann eigenvalue under the assumption that $\Sigma$ is convex (i.e.~all the principal curvatures of $\Sigma$ are non-negative). 

On the other hand, following ideas of Girouard and Polterovich in \cite{zp10, GP17}, it is shown in \cite[Proposition 2.11]{CGGS24} that given any compact Riemannian manifold with boundary, there exists a sequence of metrics with constant volume and constant boundary area along which the first non-zero Steklov eigenvalue $\sigma_1$ tends to zero. Therefore, to obtain meaningful lower bounds on $\sigma_1$, it is necessary to impose additional geometric constraints, such as a convexity condition on the boundary analogous to the one imposed by Li and Yau. The following discussion summarises some results in this direction.

It was first observed by Payne in \cite{Pay70} that for a bounded domain $\Omega\subset\mathbb{R}^2$ with boundary curvature $\kappa_{\partial\Omega}$, it holds that 
 \begin{equation*}
      \kappa_{\operatorname{\partial\Omega}} \geq \kappa >0  \implies \sg \geq \kappa,
 \end{equation*}
 with equality holding if and only if $\Omega$ is a round disc with radius $\frac{1}{\kappa}$. Escobar later initiated a systematic study of analogous problems in more general contexts in \cite{E97,E99}. First, he generalised Payne's result from Euclidean domains to all non-negatively curved compact surfaces with convex boundary:

\begin{customthm}{A}[{\cite[Theorem 1]{E97}}] \textit{Let $(M^2,g)$ be a smooth compact surface with non-empty boundary $\Sigma$. Let $K_{g}$ denote the Gaussian curvature of $(M,g)$ and $\kappa_{g}$ the geodesic curvature of $\Sigma$. Then}
 \begin{equation*}
    K_{g} \geq 0 \text{ and } \kappa_g \geq \kappa >0 \implies \sg \geq \kappa,
 \end{equation*} 
 \textit{and equality holds if and only if  $M$ is a Euclidean ball with radius $\frac{1}{\kappa}$.}
\end{customthm}

In higher dimensions Escobar also proved the following non-sharp estimate as an application of Reilly's formula \cite{Rei77}:

\begin{customthm}{B}[{\cite[Theorem 8]{E97}}]\label{thmEsc} \textit{Let $(M^{n+1},g)$ be a smooth compact manifold with non-empty boundary $\Sigma^n$ ($n\geq 2$). Let $\operatorname{Ric}_g$ denote the Ricci curvature of $(M,g)$ and $\kappa_i$ ($1\leq i \leq n$) the principal curvatures of $\Sigma$. Then}

\begin{equation*}
    \operatorname{Ric}_g \geq 0 \text{ and } \kappa_i \geq \kappa >0 \implies \sg > \frac{\kappa}{2}.
 \end{equation*} 
\end{customthm}

In later work \cite{E99}, Escobar conjectured that under the hypotheses of Theorem \ref{thmEsc}, the estimate $\sg \geq \kappa$ should hold:

\begin{customconj}{\!\!\!}[\cite{E99}]\label{conj}
	\textit{Let $(M^{n+1},g)$ be a smooth compact manifold with non-empty boundary $\Sigma^n$ ($n\geq 2$). Then}
	\begin{equation*}
	\operatorname{Ric}_g \geq 0 \text{ and } \kappa_i \geq \kappa >0 \implies \sg \geq \kappa,
	\end{equation*}
	\textit{with equality holding if and only if $M$ is a Euclidean ball with radius $\frac{1}{\kappa}$.}
\end{customconj}

\begin{rmk}
    Escobar's conjecture may be thought of as a counterpart to Yau's conjecture on the first eigenvalue of a closed minimal hypersurface in the sphere \cite{Yau82} and Fraser-Li's conjecture on the first Steklov eigenvalue of a free boundary minimal surface in the ball \cite{FL14}. 
\end{rmk}

Escobar's conjecture provides much of the motivation for our work in this paper. We note that the conjecture was recently settled under the stronger assumption that $(M,g)$ has non-negative sectional curvature by Xia and Xiong in \cite{XX23}, who showed
\begin{align*}
    \operatorname{Sec}_g \geq 0 \text{ and } \kappa_i \geq \kappa >0 \implies \sg \geq \kappa
\end{align*}
with equality holding if and only if $\Sigma$ is a Euclidean ball with radius $\frac{1}{\kappa}$. The authors utilised a weighted Reilly formula of Qiu and Xia \cite{QX15}, and used non-negativity of the sectional curvature to apply the Hessian comparison theorem in their argument. However, as far as we are aware, for $n\geq 2$ the best lower bound under the original hypotheses of Escobar's conjecture is given by Theorem \ref{thmEsc} above. Our first main result in this paper yields an improvement on the lower bound $\sigma_1 > \frac{\kappa}{2}$ under the original hypotheses of Escobar's conjecture, as we discuss now. 

\subsection{Overview of results}

Our convention is that the second fundamental form $A_g$ and corresponding shape operator $g^{-1}A_g$ of $\Sigma$ are defined with respect to the outward unit normal $\nu$ by
\begin{align*}
A_g(X,Y) \defeq g(\nabla_X \nu, Y) \quad \text{and} \quad (g^{-1}A_g)X = \nabla_X \nu.
\end{align*}
The principal curvatures $\kappa_1,\dots,\kappa_n$ of $\Sigma$ are the eigenvalues of $g^{-1}A_g$, and the mean curvature of $\Sigma$ is defined by $H = \operatorname{tr}(g^{-1}A_g) = \kappa_1 + \dots + \kappa_n$.

To state precisely how our estimates depend on various geometric quantities, we first introduce some notation. Firstly, for $\beta \geq 0, \mathcal{K}>0$ and $\kappa\in\mathbb{R}$ we define
 \begin{align*}
	E(\delta) & = E_{\beta,\mathcal{K}}(\delta) \defeq 
	\begin{cases}
	n\frac{\beta^2\tan(\beta \delta) + \beta\mathcal{K} }{\beta -\mathcal{K} \tan(\beta \delta)}+\delta^{-1} & \text{for }0<\delta<\frac{1}{\beta}\arctan(\frac{\beta}{\mathcal{K}}) \quad \text{ if }\beta>0 \\[5pt]
	\frac{n\mathcal{K}}{1-\mathcal{K} \delta} + \delta^{-1} & \text{for }0<\delta<\frac{1}{\mathcal{K}}  \qquad \qquad \quad \,\,\text{ if }\beta=0,
	\end{cases} \\[2pt]
 F(\delta)& = F_{\beta,\kappa,\mathcal{K}}(\delta) \defeq 2E(\delta) - n\kappa.
	\end{align*}
For $r>0$ we also define $M_r \defeq \{x\in M: \operatorname{dist}_g(x,\Sigma) \leq r\}$ and
\begin{align*}
\delta_{r,\beta,\mathcal{K}} \defeq \begin{cases}
\min\{r,\frac{1}{\beta}\arctan(\frac{\beta}{\mathcal{K}})\} & \text{if }\beta>0 \\
\min\{r, \frac{1}{\mathcal{K}}\} & \text{if }\beta=0.
\end{cases}
\end{align*}
 Our first main result is as follows (we refer the reader to Section \ref{S2} for the definition of the rolling radius): 
	
	\begin{thm}\label{A}
	Let $(M^{n+1},g)$ be a smooth compact manifold with non-empty boundary $\Sigma^n$, and let $R>0$ denote the rolling radius of $\Sigma$ in $M$.  Suppose that:
	\begin{enumerate}
		\item $\operatorname{Ric}_g \geq 0$ in $M$,
		\item $\operatorname{Ric}_g \geq a^2 $ and $\operatorname{Sec}_g \leq \beta^2$ in $M_r$ for constants $a\geq 0$, $\beta\geq 0$ and $r\leq R$,
		\item $0<\kappa\leq \kappa_i \leq \mathcal{K}$ on $\Sigma$ for each $i$ for constants $\kappa$ and $\mathcal{K}$.
	\end{enumerate}
	Then for any $0<\delta<\delta_{r,\beta,\mathcal{K}}$ it holds that
		\begin{align}\label{10}
		\sigma_1 \geq \frac{\kappa}{2} +\frac{-F(\delta) + \sqrt{F(\delta)^2 + 8n\kappa^2 + 16a^2}}{8} 
		\end{align}
	\end{thm}

\begin{rmk}
	One may optimise the estimate \eqref{10} by choosing $\delta\in(0,\delta_{r,\beta,\mathcal{K}})$ which minimises $F(\delta)$, or equivalently minimises $E(\delta)$. 
\end{rmk}

\begin{rmk}\label{35}
 We point out that the only \textit{assumptions} made in Theorem \ref{A} are $\operatorname{Ric}_g \geq 0$ and $0<\kappa\leq \kappa_i$, whilst positivity of the rolling radius and bounds of the form $\kappa_i\leq \mathcal{K}$ and $\operatorname{Sec}_g \leq \beta^2$ hold by compactness. Moreover, in Theorem \ref{A} our estimate does not depend on any lower bound for the sectional curvature. However, the dependence on the rolling radius can be replaced with dependence on a lower bound for the sectional curvature -- see Corollary \ref{B} in Section \ref{100}.
\end{rmk}

\begin{rmk}
	When $a=\beta=0$ (in which case $(M,g)$ is flat near $\Sigma$), none of the dimensional or universal constants appearing in \eqref{10} can be improved whilst keeping the others fixed -- see Remark \ref{101} in Section \ref{100}.
\end{rmk}

Our next main result concerns the special case $a>0$, i.e.~when the Ricci curvature is positive near $\Sigma$. In this case, using similar ideas to the proof of Theorem \ref{A} (which will be discussed in Section \ref{S2}), we may relax the previous strict convexity assumptions on $\Sigma$ and still obtain a positive lower bound for $\sigma_1$:

\begin{thm}\label{E}
	Let $(M^{n+1},g)$ be a smooth compact manifold with non-empty boundary $\Sigma^n$, and let $R>0$ denote the rolling radius of $\Sigma$ in $M$. Suppose that:
	\begin{enumerate}
		\item $\operatorname{Ric}_g \geq 0$ in $M$,
		\item $\operatorname{Ric}_g \geq a^2 >0$ and $\operatorname{Sec}_g \leq \beta^2$ in $M_r$ for constants $a> 0$, $\beta > 0$ and $r\leq R$,
		\item $H\geq 0$ and 
		\begin{align}\label{32}
		-\bigg(\frac{-E(\delta)+\sqrt{E(\delta)^2+4a^2}}{2}\bigg) < \kappa \leq \kappa_i \leq \mathcal{K} 
		\end{align}
		on $\Sigma$ for each $i$ for constants $\kappa, \mathcal{K}$ and some $0<\delta<\delta_{r,\beta,\mathcal{K}}$.
	\end{enumerate}
	If $0<\delta<\delta_{r,\beta,\mathcal{K}}$ is such that \eqref{32} holds, then
	\begin{align}\label{34}
	\sigma_1 \geq \frac{1}{2}\bigg(\kappa + \frac{-E(\delta)+\sqrt{E(\delta)^2+4a^2}}{2} \bigg). 
	\end{align}
\end{thm}

\begin{rmk}
	Condition 3 in Theorem \ref{E} is automatically satisfied if $\Sigma$ is convex (that is if $g^{-1}A_g$ is non-negative definite), but also allows more generally for some of the principal curvatures to be negative. In comparison, the hypotheses of Theorem \ref{A} requires $\Sigma$ to be strictly convex. 
\end{rmk}

\begin{rmk}
	If we assume in addition to the hypotheses of Theorem \ref{E} that $\Sigma$ is strictly mean-convex, i.e. $H >0$, then the estimate \eqref{34} can be improved by means of an iteration argument that is used in the proof of Theorem \ref{A} -- see Section \ref{S33} for the details. 
\end{rmk}

We next address the problem of obtaining a lower bound for $\sigma_1$ similar to \eqref{10} but depending on weaker geometric bounds. In this direction, under an additional assumption relating the lower bounds on the principal curvatures and sectional curvatures, we are able to replace the dependence on a sectional curvature upper bound in Theorem \ref{A} with a Ricci curvature upper bound, and the dependence on a principal curvature upper bound with a mean curvature upper bound. In what follows, for $ b \geq 0, \mathcal{H}>0$ and $\kappa\in\mathbb{R}$ we define
	\begin{align*}
P(\delta) &= P_{ b ,\mathcal{H}}(\delta) \defeq 
\begin{cases}
\frac{ b^2\tan( b  \delta) +  b  \mathcal{H} }{ b  - \mathcal{H} \tan( b  \delta)}+\delta^{-1} & \text{for }0<\delta<\frac{1}{ b }\arctan(\frac{ b }{\mathcal{H}}) \quad \text{ if } b >0 \\[5pt]
\frac{\mathcal{H}}{1-\mathcal{H} \delta} + \delta^{-1} & \text{for }0<\delta<\frac{1}{\mathcal{H}}  \qquad \qquad \quad \,\,\text{ if } b =0,
\end{cases} \\[2pt]
Q(\delta) & = Q_{b, \kappa ,\mathcal{H}}(\delta) \defeq 2P(\delta) - n\kappa.
\end{align*}

\begin{thm}\label{C}
	Let $(M^{n+1},g)$ be a smooth compact manifold with non-empty boundary $\Sigma^n$, and let $R>0$ denote the rolling radius of $\Sigma$ in $M$. Suppose that:
	\begin{enumerate}
		\item $\operatorname{Ric}_g \geq 0$ in $M$,
		\item $0<\kappa\leq \kappa_i$ for each $i$ and $H \leq \mathcal{H}$ on $\Sigma$ for constants $\kappa$ and $\mathcal{H}$,
		\item $a^2 \leq \operatorname{Ric}_g \leq  b^2$ and $\operatorname{Sec}_g \geq -\alpha^2$ in $M_r$ for constants $a,b \geq 0$, $0\leq\alpha \leq \kappa$ and $r\leq R$.
	\end{enumerate}
Then for any $0<\delta< \delta_{r, b ,\mathcal{H}}$ it holds that
	\begin{align*}
	\sigma_1 \geq \frac{\kappa}{2} +\frac{-Q(\delta) + \sqrt{Q(\delta)^2 + 8n\kappa^2 + 16a^2}}{8}.
	\end{align*}
\end{thm}

Compared with the proof Theorems \ref{A} and \ref{E}, the proof of Theorem \ref{C} utilises a new mean curvature comparison result -- we refer the reader to Theorem \ref{D} and Lemma \ref{L3} for the details. We simply remark here that the role of the condition $0\leq \alpha\leq \kappa$ in Theorem \ref{C} is to ensure that the hypersurfaces parallel to $\Sigma$ remain convex up to the rolling radius, which is used in the proof of Lemma \ref{L3}.

\begin{rmk}
	In light of Theorem \ref{E}, one may ask whether the positive lower bound on the $\kappa_i$ in Theorem \ref{C} may be weakened if one assumes $a>0$. However, in adapting Lemma \ref{L3} to this case, it seems that our methods still require each $\kappa_i \geq 0$ (or more generally, a technical 2-convexity condition on $\Sigma$) \textit{and} non-negativity of the sectional curvature near $\Sigma$. Since these geometric assumptions are very stringent, and in order to keep the paper concise, we will not consider this case any further. 
\end{rmk}

The plan of the paper is as follows. In Section \ref{S2} we prove an initial inequality using Reilly's formula, and we then give an outline of how our main results are obtained by estimating certain terms in this inequality. In Section \ref{100} we prove Theorem \ref{A} and an associated result (Corollary \ref{B}). In Section \ref{S33} we prove Theorem \ref{E} and an associated result (Theorem \ref{F}). In Section \ref{102} we establish a new mean curvature comparison result (see Theorem \ref{D} and Lemma \ref{L3}), which is then used to prove Theorem \ref{C}. Finally, in Section \ref{103} we prove a spectral gap result assuming $\operatorname{Ric}_g<0$ and sufficient convexity of the boundary.

\section{An initial lemma and summary of proofs}\label{S2}

In this section, we first derive an initial estimate using Reilly's formula in the spirit of Escobar \cite{E97}, and then give some key ideas in the proofs of our main results. The initial estimate is as follows (henceforth, $\nabla^T$ denotes the gradient operator of the induced metric on $\Sigma$):

\begin{lem}\label{L4}
	Let $(M^{n+1},g)$ be a smooth manifold with non-empty boundary $\Sigma$ and suppose: 
	\begin{enumerate}
		\item $\operatorname{Ric}_g \geq 0$ in $M$,
		\item $\operatorname{Ric}_g \geq a_1$ in $M_r$,
		\item $H \geq a_2$ and $\kappa_i \geq a_3$ for each $i$ on $\Sigma$
	\end{enumerate}
for constants $a_1\geq 0$, $a_2,a_3\in\mathbb{R}$ and $r>0$. If $f$ is an eigenfunction associated to the $j$'th Steklov eigenvalue $\sigma_j$, then
\begin{align}\label{7}
0 \geq \int_{M_r} |\nabla^2 f|^2\,dv_g +(a_1 + a_2\sigma_j)\int_{M_r} |\nabla f|^2\,dv_g + (a_3 -2\sigma_j)\int_{\Sigma}|\nabla^T f|^2\,dS_g.
\end{align}
\end{lem}
\begin{proof}
We recall Reilly's formula \cite{Rei77}, which states that for any $u\in C^\infty(M)$, 
	\begin{align}\label{25}
	\int_M \big((\Delta u)^2 - |\nabla^2 u|^2\big)\,dv_g & = \int_M \operatorname{Ric}_g(\nabla u,\nabla u)\,dv_g + \int_{\Sigma} H u_\nu^2\,dS_g \nonumber \\
	& \quad  - 2\int_{\Sigma}\langle \nabla^T u,\nabla^T u_\nu\rangle \,dS_g + \int_{\Sigma} A_g(\nabla^T u,\nabla^T u)\,dS_g,
	\end{align}
	where $u_\nu$ denotes the outward normal derivative of $u$ when restricted to the boundary. Taking $u=f$ in \eqref{25}, recalling \eqref{16} and using the geometric assumptions in the statement of the lemma, we thus obtain
	\begin{align*}
	0 \geq \int_M |\nabla^2 f|^2\,dv_g + a_1\int_{M_r} |\nabla f|^2\,dv_g + a_2\int_{\Sigma} f_\nu^2\,dS_g + (a_3-2\sigma_j)\int_{\Sigma}|\nabla^T f|^2\,dS_g.
	\end{align*}
	To obtain \eqref{7}, it remains to observe that $\int_{M_r}|\nabla^2 f|^2\,dv_g \leq \int_M |\nabla^2f|^2\,dv_g$ and 
		\begin{align*}
	\int_{M_r} |\nabla f|^2\,dv_g\leq \int_M |\nabla f|^2\,dv_g = \int_{\Sigma} f_\nu f\,dS_g = \frac{1}{\sigma_j}\int_{\Sigma}f_\nu^2\,dS_g,
	\end{align*}
	which is a simple consequence of integration by parts and \eqref{16}.
		\end{proof}
	
\begin{rmk}
	We note that Escobar obtained his estimate $\sigma_1 > \frac{\kappa}{2}$ in the case $\operatorname{Ric}_g \geq 0$ and $\kappa_i \geq \kappa>0$ (see Theorem \ref{thmEsc} in the introduction) by taking $a_1=a_2=0$ and $a_3 = \kappa$ in \eqref{7}; the remaining integral of $|\nabla^2 f|^2$ in \eqref{7} is dropped altogether.
\end{rmk} 

Let us now indicate the main ideas in the proof of Theorem \ref{A} (the proofs of our other main results are largely based on the proof of Theorem \ref{A}). In the setting of Theorem \ref{A}, the estimate \eqref{7} for $j=1$ reads 
\begin{align}\label{26}
0 \geq \int_{M_r} |\nabla^2 f|^2\,dv_g +(a^2 + n\kappa \sigma_1)\int_{M_r} |\nabla f|^2\,dv_g + (\kappa -2\sigma_1)\int_{\Sigma}|\nabla^T f|^2\,dS_g.
\end{align}
The key idea is to then obtain a lower bound for the first two integrals in \eqref{26} as a positive multiple of $\int_{\Sigma}|\nabla^T f|^2\,dS_g$, or equivalently an upper bound for $\int_{\Sigma}|\nabla^T f|^2\,dS_g$ in terms of a linear combination of $\int_{M_r} |\nabla^2 f|^2\,dv_g$ and $\int_{M_r} |\nabla f|^2\,dv_g$ with positive coefficients. Since these are integrals over a neighbourhood of the boundary $\Sigma$ in which we have control of the sectional curvature (recall we are assuming $\operatorname{Sec}_g \leq \beta^2$ in $M_r$), we can control the mean curvature of the hypersurfaces parallel to $\Sigma$, at least after restricting to a possibly smaller neighbourhood of $\Sigma$ in which the parallel hypersurfaces are known to remain smooth and embedded -- see Lemma \ref{L1}. (Our argument here invokes results from comparison geometry established in a related context in \cite{CGH20} -- note, however, that new comparison results are needed in the proof of Theorem \ref{C}, and these are established in Section \ref{102}.) These mean curvature bounds are then used to obtain the desired integral estimates, by integrating the quantity $|\nabla f|^2\Delta d$ over a neighbourhood of $\Sigma$ and using the fact that $-\Delta d$ is the mean curvature of a parallel hypersurface -- see Lemma \ref{L2}. (Our argument here is inspired by the proof \cite[Lemma 3.6]{DSS23}, wherein the first author, Sire \& Spruck obtain a related lower bound in the context of eigenvalue estimates for closed minimal hypersurfaces in the sphere.) Once an initial improvement on the lower bound $\sigma_1 > \frac{\kappa}{2}$ is obtained using the aforementioned argument, an iteration procedure then yields the bound in Theorem \ref{A}. 

The size of the neighbourhood of $\Sigma$ in which one can obtain the mean curvature bounds mentioned above is limited by the so-called rolling radius $R$ of $\Sigma$ in $M$, defined to be the distance from $\Sigma$ to its cut locus in $M$. It is well-known that for any $\delta<R$, the exponential map of the normal bundle of $\Sigma$ defines a diffeomorphism between $\Sigma\times[0,\delta]$ and $M_\delta$, and that $M_\delta$ is foliated by smooth parallel hypersurfaces $\Sigma_s$ for $s\in[0,\delta]$, each of which is diffeomorphic to $\Sigma$ -- we refer e.g.~to \cite[Chapter 3]{Gray04} for proofs of these facts.

	\section{Proofs of main results}

\subsection{Proof of Theorem \ref{A} and Corollary \ref{B}}\label{100}

In this section we prove Theorem \ref{A}. Following the outline of the proof given in the previous section, we start with the following mean curvature bound for parallel hypersurfaces:

\begin{lem}\label{L1}
	Let $(M^{n+1},g)$ be a smooth compact manifold with non-empty boundary $\Sigma^n$ and let $R>0$ denote the rolling radius of $\Sigma$ in $M$. Suppose
	\begin{enumerate}
		\item $\kappa_i \leq \mathcal{K}$ on $\Sigma$ for each $i$ for some $\mathcal{K}>0$, 
		\item $\operatorname{Sec}_g \leq \beta^2$ in $M_r$ for some $\beta\geq 0$ and $r\leq R$. 
	\end{enumerate}
Then for each $0<\delta<\delta_{r,\beta,\mathcal{K}}$, the mean curvature of the parallel hypersurface $\Sigma_\delta$ satisfies the upper bound
\begin{align*}
H_{\Sigma_\delta} \leq E_{\beta,\mathcal{K}}(\delta) - \delta^{-1}. 
\end{align*}
\end{lem}

\begin{proof}
	Consider the ODE
\begin{align*}
\begin{cases}
\phi'(t) + \phi(t)^2 + \beta^2 = 0 \\
\phi(0) = -\mathcal{K} . 
\end{cases}
\end{align*}
A direct computation yields the solution 
\begin{align*}
\phi(t) = \begin{cases}
-\frac{\beta^2\tan(\beta t) + \beta\mathcal{K} }{\beta -\mathcal{K} \tan(\beta t)} & \text{if }\beta>0 \\[5pt]
-\frac{\mathcal{K}}{1-\mathcal{K} t} &\text{if }\beta=0,
\end{cases}
\end{align*}
which exists on the maximal interval $t\in[0,\beta^{-1}\arctan(\beta\mathcal{K}^{-1}))$ if $\beta>0$ and $t\in[0,\frac{1}{\mathcal{K}})$ if $\beta=0$. By the principal curvature comparison theorem \cite[Theorem 12]{CGH20}, for each $0<\delta<\delta_{r,\beta,\mathcal{K} }$ the principal curvatures $\kappa_i(\delta)$ of the parallel hypersurface $\Sigma_\delta$ satisfy the upper bound 
\begin{align*}
\kappa_i(\delta) \leq - \phi(\delta).
\end{align*}
Therefore, for each $0<\delta<\delta_{r,\beta,\mathcal{K} }$ the mean curvature of $\Sigma_\delta$ satisfies the upper bound
\begin{align*}
H_{\Sigma_\delta} \leq  \begin{cases}
n\frac{\beta^2\tan(\beta \delta) + \beta\mathcal{K} }{\beta -\mathcal{K} \tan(\beta \delta)} & \text{if }\beta>0 \\[5pt]
\frac{n\mathcal{K}}{1-\mathcal{K} \delta} & \text{if }\beta=0,
\end{cases}
\end{align*}
that is $H_{\Sigma_\delta} \leq E_{\beta,\mathcal{K}}(\delta) - \delta^{-1}$. 
\end{proof}

We now use the mean curvature bound in Lemma \ref{L1} to obtain our main integral estimate in the proof of Theorem \ref{A}: 

\begin{lem}\label{L2}
	Under the same hypotheses as Lemma \ref{L1}, for each $\ep>0$, $0<\delta< \delta_{r,\beta,\mathcal{K} }$ and $u\in C^\infty(M)$ it holds that
\begin{align}\label{11}
\ep\int_{\Sigma}|\nabla^T u|^2\,dS_g \leq \int_{M_\delta}|\nabla^2 u|^2\,dv_g + \ep \big(\ep+E_{\beta,\mathcal{K}}(\delta)\big)\int_{M_\delta}|\nabla u|^2\,dv_g.
\end{align}
\end{lem}

\begin{proof}
	Fix $\delta<\delta_{r,\beta,\mathcal{K}}$, let $d$ denote the distance to $\Sigma$, and recall that the mean curvature of a parallel hypersurface is given by $-\operatorname{div}\nabla d = -\Delta d$. It therefore follows from Lemma \ref{L1} that for $t\in[0,\delta]$,
\begin{align}\label{5}
-\int_{M_t} |\nabla u|^2\Delta d\,dv_g \leq \big(E_{\beta,\mathcal{K}}(\delta)-\delta^{-1}\big)\int_{M_t}|\nabla u|^2\,dv_g. 
\end{align}
On the other hand, integration by parts gives
\begin{align}\label{4}
-\int_{M_t} |\nabla u|^2\Delta d\,dv_g = \int_{M_t}\langle\nabla d,\nabla|\nabla u|^2\rangle\,dv_g - \int_{\partial M_t} |\nabla u|^2\langle \nabla d,\nu\rangle\,dS_g 
\end{align}
where $\nu$ is the outward unit normal to $M_t$. Since $\langle\nabla d,\nu\rangle = -1$ on $\Sigma =  \partial M$ and $\langle \nabla d,\nu\rangle = 1$ on $\Sigma_t \defeq \partial M_t \backslash \partial M$, \eqref{4} yields
\begin{align*}
-\int_{M_t} |\nabla u|^2\Delta d\,dv_g & = \int_{M_t}\langle\nabla d,\nabla|\nabla u|^2\rangle\,dv_g + \int_{\Sigma} |\nabla u|^2\,dS_g - \int_{\Sigma_t}|\nabla u|^2\,dS_g \nonumber \\
& \geq -2\int_{M_t}|\nabla u||\nabla^2 u|\,dv_g + \int_{\Sigma} |\nabla u|^2\,dS_g - \int_{\Sigma_t}|\nabla u|^2\,dS_g \nonumber \\
& \geq -\ep\int_{M_t}|\nabla u|^2\,dv_g - \ep^{-1}\int_{M_t}|\nabla^2 u|^2\,dv_g + \int_{\Sigma} |\nabla u|^2\,dS_g - \int_{\Sigma_t}|\nabla u|^2\,dS_g 
\end{align*}
for any $\ep>0$, the last line following from Young's inequality. Substituting this into \eqref{5} and rearranging, we obtain 
\begin{align*}
\int_{\Sigma}|\nabla u|^2\,dS_g & \leq \int_{\Sigma_t}|\nabla u|^2\,dS_g + \big(E_{\beta,\mathcal{K}}(\delta) - \delta^{-1}+\ep\big)\int_{M_t}|\nabla u|^2\,dv_g + \ep^{-1}\int_{M_t}|\nabla^2 u|^2\,dv_g \nonumber \\
& \leq \int_{\Sigma_t}|\nabla u|^2\,dS_g + \big(E_{\beta,\mathcal{K}}(\delta) - \delta^{-1}+\ep\big)\int_{M_\delta}|\nabla u|^2\,dv_g + \ep^{-1}\int_{M_\delta}|\nabla^2 u|^2\,dv_g 
\end{align*}
for $t\in[0,\delta]$ and any $\ep>0$. We now integrate over $t\in[0,\delta]$ to obtain
\begin{align*}
\delta \int_{\Sigma}|\nabla u|^2\,dS_g & \leq \int_{M_\delta}|\nabla u|^2\,dv_g + \delta\big(E_{\beta,\mathcal{K}}(\delta) - \delta^{-1}+\ep\big)\int_{M_\delta}|\nabla u|^2\,dv_g + \delta\ep^{-1}\int_{M_\delta}|\nabla^2 u|^2\,dv_g \nonumber \\
& =  \delta\big(E_{\beta,\mathcal{K}}(\delta)+\ep\big)\int_{M_\delta}|\nabla u|^2\,dv_g + \delta\ep^{-1}\int_{M_\delta}|\nabla^2 u|^2\,dv_g.
\end{align*}
Therefore
\begin{align*}
\int_{\Sigma}|\nabla^T u|^2\,dS_g \leq \int_\Sigma |\nabla u|^2\,dS_g \leq \big(E_{\beta,\mathcal{K}}(\delta)+\ep\big)\int_{M_\delta}|\nabla u|^2\,dv_g + \ep^{-1}\int_{M_\delta}|\nabla^2 u|^2\,dv_g,
\end{align*}
which completes the proof of the lemma after multiplying through by $\ep$.
\end{proof}

We are now in a position to complete the proof of Theorem \ref{A}:

\begin{proof}[Proof of Theorem \ref{A}]
Let $f$ be an eigenfunction for $\sigma_1$. Recalling \eqref{26} and the fact $\sigma_1 > \frac{\kappa}{2}$ and $\delta_{r,\beta,\mathcal{K}}\leq r$, we have for each $0<\delta<\delta_{r,\beta,\mathcal{K}}$ the estimate
\begin{align}\label{8}
0 \geq \int_{M_\delta} |\nabla^2 f|^2\,dv_g + \bigg(a^2 + \frac{n\kappa^2}{2}\bigg)\int_{M_\delta} |\nabla f|^2\,dv_g + (\kappa-2\sigma_1)\int_{\Sigma}|\nabla^T f|^2\,dS_g.
\end{align}
We now substitute the lower bound for $\int_{M_\delta} |\nabla^2 f|^2\,dv_g$ given by Lemma \ref{L2} (taking $u=f$ therein) into \eqref{8}:  
\begin{align}\label{12}
0 \geq \bigg(a^2+\frac{n\kappa^2}{2} - \ep^2 - \ep E\bigg)\int_{M_\delta} |\nabla f|^2\,dv_g + (\kappa+\ep - 2\sigma_1)\int_{\Sigma}|\nabla^T f|^2\,dS_g.
\end{align}
This yields the lower bound $\sigma_1 \geq \frac{\kappa+\ep}{2}$ as long as $\ep$ is chosen such that the coefficient of the first integral in \eqref{12} is non-negative. The optimal (i.e.~largest) choice of $\epsilon$ is the one making this coefficient zero, that is
\begin{align*}
\ep = \ep_1 \defeq  \frac{-E + \sqrt{E^2 + 2n\kappa^2 + 4a^2}}{2}.
\end{align*}

We now repeat the above procedure. In place of \eqref{8}, we instead substitute the improved lower bound $\sigma_1 \geq \frac{\kappa}{2} + \frac{\ep_1}{2}$ into \eqref{26} to obtain
\begin{align}\label{8'}
0 \geq \int_{M_\delta} |\nabla^2 f|^2\,dv_g + \bigg(a^2 + \frac{n\kappa^2}{2} + \frac{n\kappa\ep_1}{2}\bigg)\int_{M_\delta} |\nabla f|^2\,dv_g + (\kappa-2\sigma_1)\int_{\Sigma}|\nabla^T f|^2\,dS_g.
\end{align} 
Again we substitute the lower bound for $\int_{M_\delta} |\nabla^2 f|^2\,dv_g$ given by Lemma \ref{L2} into \eqref{8'}, leading to the lower bound $\sigma_1 \geq \frac{\kappa}{2} + \frac{\ep_2}{2}$ where
\begin{align*}
 \ep_2 \defeq \frac{-E + \sqrt{E^2 + 2n\kappa^2 + 4a^2 + 2n\kappa \ep_1}}{2}.
\end{align*}
Continuing in this way, we obtain the lower bound $\sigma_1 \geq \frac{\kappa}{2} + \frac{\mathcal{E}}{2}$ where $\mathcal{E}$ is defined to be the limit as $i\rightarrow\infty$ of the following iterative procedure:
\begin{align}\label{33}
\ep_{i+1} = \frac{-E + \sqrt{E^2 + 2n\kappa^2 + 4a^2 + 2n\kappa \ep_i}}{2}, \qquad \ep_0 = 0. 
\end{align}
One can find $\mathcal{E}$ explicitly by setting $\ep_{i+1} = \ep_i =x$ in \eqref{33} and solving for $x$: 
\begin{align*}
\mathcal{E} = \frac{n\kappa - 2E + \sqrt{(n\kappa-2E)^2 + 8(n\kappa^2 + 2a^2)}}{4} = \frac{-F + \sqrt{F^2 + 8n\kappa^2+16a^2}}{4}.
\end{align*}
This completes the proof of Theorem \ref{A}.
\end{proof}

\begin{rmk}\label{101}
	Consider the special case that $M$ is the unit ball in $\mathbb{R}^n$ with the flat metric, so that $a=\beta=0$ and $\kappa=\mathcal{K} = 1$. Then it is routine to check that $E(\delta) = \frac{n}{1-\delta} + \delta^{-1}$ is maximised at $\delta = \frac{\sqrt{n}-1}{n-1}$, hence $F(\delta)$ is maximised at $\delta = \frac{\sqrt{n}-1}{n-1}$ with $F(\frac{\sqrt{n}-1}{n-1})=2 + 4\sqrt{n} + n$. Substituting this into \eqref{10} yields
	\begin{align*}
	\sigma_1 \geq \frac{1}{2} + \frac{-(2+4\sqrt{n} +n) + \sqrt{(2+4\sqrt{n}+n)^2 + 8n}}{8} \longrightarrow 1 \quad \text{as }n\rightarrow\infty. 
	\end{align*}
	Since $\sigma_1 =1$ when $M$ is the unit ball in $\mathbb{R}^n$, the above computation shows that when $a=\beta=0$, none of the dimensional or universal constants in \eqref{10} can be improved, else for sufficiently large values of $n$ the above computation would yield $\sigma_1 >1$. 
\end{rmk}

As pointed out in Remark \ref{35}, our estimate in Theorem \ref{A} does not depend on any lower bound for the sectional curvature. However, such a lower bound (which always exists by compactness) can be used to obtain a lower bound for the rolling radius of $\Sigma$ in $M$ as follows. Let $\alpha>0$ be a constant such that $\operatorname{Sec}_g \geq -\alpha^2$ in $M$ (we do not consider the case $\alpha=0$, since Escobar's conjecture is already known in this setting by \cite{XX23}), and define
\begin{align*}
F(\kappa,\alpha^2) = \begin{cases}
\frac{1}{\alpha}\operatorname{arctanh}(\frac{\kappa}{\alpha}) & \text{if }0<\kappa<\alpha \\
+ \infty & \text{if }\kappa \geq \alpha.
\end{cases}
\end{align*}
Then by \cite[Theorem 3.11]{DL91}, strict convexity of $\Sigma$ coupled with the bound $-\alpha^2 \leq \operatorname{Sec}_g \leq \beta^2$ in $M$ implies $R \geq \min\{F(\kappa,\alpha^2), \frac{1}{\beta}\arctan(\frac{\beta}{\mathcal{K}})\}$ when $\beta >0$ and $R \geq \min\{F(\kappa,\alpha^2), \frac{1}{\mathcal{K}}\}$ when $\beta=0$. In other words, $R \geq \delta_{F(\kappa,\alpha^2),\beta,\mathcal{K}}$. Thus we obtain:
\begin{cor}\label{B}
	Let $(M^{n+1},g)$ be a smooth compact manifold with non-empty boundary $\Sigma^n$, and suppose in addition to Assumptions 1 and 3 in Theorem \ref{A} that
	\begin{enumerate}
		\item[3$\,'$.] $-\alpha^2 \leq \operatorname{Sec}_g \leq \beta^2$ in $M$ for constants $\alpha>0$ and $\beta\geq 0$, and $\operatorname{Ric}_g \geq a^2 \geq 0$ in $M_r$ for some $r \leq \delta_{F(\kappa,\alpha^2),\beta,\mathcal{K}}$. 
	\end{enumerate}
	Then the lower bound \eqref{10} holds for any $0<\delta<r$.
\end{cor}

\subsection{Proof of Theorem \ref{E} and Theorem \ref{F}}\label{S33}

In this section we start with the proof of Theorem \ref{E}:

\begin{proof}[Proof of Theorem \ref{E}]
	The proof closely follows that of Theorem \ref{A}, except that we do not have the initial estimate $\sigma_1>\frac{\kappa}{2}$ to work with. 
	
	We start by applying Lemma \ref{L4} with $a_1 = \kappa $, $a_2 = 0$ and $a_3 = a^2$, which yields
	\begin{align}\label{31}
	0 \geq \int_{M_r} |\nabla^2 f|^2\,dv_g + a^2 \int_{M_r} |\nabla f|^2\,dv_g + (\kappa -2\sigma_1)\int_{\Sigma}|\nabla^T f|^2\,dS_g.
	\end{align}
	On the other hand, by Lemma \ref{L2}, for each $\ep>0$ and $0<\delta<\delta_{r,\beta,\mathcal{K}}$, we have the lower bound
	\begin{align}\label{36}
	\int_{M_\delta}|\nabla^2 f|^2\,dv_g \geq -\ep(\ep+E(\delta))\int_{M_\delta}|\nabla f|^2\,dv_g + \ep\int_{\Sigma}|\nabla^T f|^2\,dS_g. 
	\end{align}
	Substituting this back into \eqref{31} (recalling that $\delta_{r,\beta,\mathcal{K}}\leq r$) yields 
	\begin{align*}
	0 \geq (a^2 - \ep^2 - \ep E)\int_{M_\delta}|\nabla f|^2\,dv_g + (\ep + \kappa -2\sigma_1)\int_{\Sigma}|\nabla^T f|^2\,dS_g. 
	\end{align*}
	Taking $\ep = \frac{-E+\sqrt{E^2+4a^2}}{2}$ (so that $a^2 - \ep^2 - \ep E = 0$) and choosing $\delta<\delta_{r,\beta,\mathcal{K}}$ such that \eqref{32} holds, the desired lower bound
	\begin{align}\label{34'}
	\sigma_1 \geq \frac{1}{2}\bigg(\kappa  + \frac{-E(\delta)+\sqrt{E(\delta)^2+4a^2}}{2} \bigg). 
	\end{align}
	follows.
\end{proof}

We point out that since we do not assume a positive lower bound of the form $H \geq  h  > 0$ in Theorem \ref{E}, we cannot iterate the estimate \eqref{34} as we did in the proof of Theorem \ref{A}. If we do assume such a bound, then we have the following: 
\begin{thm}\label{F}
	Assume in addition to the hypotheses of Theorem \ref{E} that $H \geq  h >0$, and denote $T(\delta) = 2E(\delta) -  h $. If $0<\delta<\delta_{r,\beta,\mathcal{K}}$ is such that \eqref{32} holds, then
	\begin{align}\label{39}
	\sigma_1 \geq \frac{\kappa }{2} + \frac{-T(\delta) + \sqrt{T(\delta)^2 + 8 h \kappa  + 16a^2}}{8}.
	\end{align}
\end{thm}

\begin{proof}
	Lemma \ref{L4} with $a_1=\kappa $, $a_2 =  h $ and $a_3 = a^2$ yields
	\begin{align}\label{37}
	0 \geq \int_{M_r} |\nabla^2 f|^2\,dv_g + (a^2+ h \sigma_1) \int_{M_r} |\nabla f|^2\,dv_g + (\kappa -2\sigma_1)\int_{\Sigma}|\nabla^T f|^2\,dS_g,
	\end{align}
	and substituting \eqref{36} into \eqref{37} thus gives
	\begin{align}\label{38}
	0 \geq (a^2 +  h \sigma_1 - \ep^2 - \ep E)\int_{M_\delta}|\nabla f|^2\,dv_g + (\ep + \kappa -2\sigma_1)\int_{\Sigma}|\nabla^T f|^2\,dS_g.
	\end{align}
	Let us denote the RHS of \eqref{34'} as $\frac{1}{2}(\kappa  + \ep_1)$. Substituting the lower bound $\sigma_1 \geq \frac{1}{2}(\kappa  + \ep_1)$ into the term $ h \sigma_1$ in \eqref{38}, and subsequently choosing 
	\begin{align*}
	\ep = \ep_2 \defeq \frac{-E + \sqrt{E^2 + 4a^2 + 2 h (\kappa +\ep_1)}}{2},
	\end{align*}
	we obtain the lower bound $\sigma_1 \geq \frac{1}{2}(\kappa +\ep_2)$ in the same way. Continuing this procedure, we obtain the lower bound $\sigma_1 \geq \frac{1}{2}(\kappa  + \mathcal{E})$, where $\mathcal{E}$ is defined to the be the limit as $i\rightarrow\infty$ of the following iterative procedure: 
	\begin{align*}
	\ep_{i+1} = \frac{-E + \sqrt{E^2+4a^2 + 2 h (\kappa +\ep_i)}}{2}, \qquad \ep_1 = \frac{-E + \sqrt{E^2 + 4a^2}}{2}.
	\end{align*}
	It is routine to check that the limit $\mathcal{E}$ is given by 
	\begin{align*}
	\mathcal{E}  = \frac{ h  - 2E + \sqrt{( h -2E)^2 + 8  h \kappa  + 16a^2}}{4},
	\end{align*}
	which proves \eqref{39}.
\end{proof}

\subsection{Proof of Theorem \ref{C}}\label{102}

In this section we prove Theorem \ref{C}. A key tool is the following mean curvature comparison theorem, which is a counterpart to \cite[Theorem 16]{CGH20} but under an additional convexity assumption. We recall that the boundary $\Sigma$ of $M$ is called \textit{2-convex} if it is both mean-convex ($H\geq 0$) and the second elementary symmetric polynomial of its shape operator is non-negative:
\begin{align*}
S_2(g^{-1}A_g) \defeq \sum_{i<j}\kappa_i\kappa_j = \frac{1}{2}\Big[\big(\operatorname{tr}(g^{-1}A_g)\big)^2 - \operatorname{tr}\big((g^{-1}A_g)^2\big)\Big]\geq 0.
\end{align*}
When $n\leq 2$, $\Sigma$ is 2-convex if and only if it is convex, but in higher dimensions 2-convexity is a weaker notion than convexity. 

\begin{thm}[Mean curvature comparison theorem]\label{D}
	Let $(M^{n+1},g)$ be a smooth compact manifold with non-empty boundary $\Sigma^n$, and let $R>0$ denote the rolling radius of $\Sigma$ in $M$. Suppose that for some $r\leq R$: 
	\begin{enumerate}
		\item $\operatorname{Ric}_g \leq  b^2$ in $M_r$ for some $ b \geq 0$, 
		\item $H\leq \mathcal{H}$ on $\Sigma$ for some $\mathcal{H}>0$,
		\item The parallel hypersurfaces $\Sigma_\delta$ are 2-convex for $0 \leq \delta < r$.
	\end{enumerate}
Then for each $0<\delta< \delta_{r, b ,\mathcal{H}}$, the mean curvature of $\Sigma_\delta$ satisfies
	\begin{align}\label{22}
	H_{\Sigma_\delta} \leq P_{ b ,\mathcal{H}}(\delta) - \delta^{-1}.
	\end{align}
\end{thm}

\begin{proof}
	Given $p_0\in \Sigma$, the exponential map of the normal bundle of $\Sigma$ defines a geodesic curve $ \gamma_{p_0} :\mathbb{R}_+\rightarrow M$ which is normal to $\Sigma$ at $p_0$. Denote $E_s(p_0) \defeq  \gamma_{p_0}(s)$. As explained in \cite[Section 2.3]{CGH20} (we summarise the main points here and refer to \cite{CGH20} for the details), given $v\in T_{p_0}\Sigma$ and a smooth curve $p:(-\ep,\ep)\rightarrow\Sigma$ satisfying $p(0) = p_0$ and $p'(0) = v$, 
	\begin{align*}
	J(s) \defeq \frac{d}{dt}\bigg|_{t=0} E_s(p(t)) \in T\Sigma_s 
	\end{align*}
	defines a Jacobi field along the normal geodesic $ \gamma_{p_0}(s)$, and thus the shape operator $S_s$ of $\Sigma_s$ satisfies the Riccati equation 
	\begin{align}\label{20}
	S_s'(J(s)) + S_s^2(J(s)) + R_{\nabla d}(J(s))= 0
	\end{align}
	where $R_{\nabla d}(J(s)) \defeq \operatorname{Riem}(J(s),\nabla d)\nabla d$ is the curvature tensor of $M$ in the normal direction $\nabla d$. After identifying $T_{ \gamma_{p_0}(s)}\Sigma_s$ with $T_{p_0}\Sigma$ via parallel transport along $ \gamma $, we can view $S_s$ and $R_{\nabla d}$ as endomorphisms on $T_{p_0}\Sigma$, which by \eqref{20} then satisfy the following Riccati equation as endomorphisms on $T_{p_0}\Sigma$:
	\begin{align*}
	S'(s) + S^2(s) + R(s) = 0. 
	\end{align*}
 In particular, 
 \begin{align}\label{21}
     \operatorname{tr}S'(s) + \operatorname{tr}S^2(s) + \operatorname{tr}R(s) = 0.
 \end{align}
 
	We now appeal to 2-convexity of the parallel hypersurfaces to assert that $\operatorname{tr}S^2 \leq (\operatorname{tr}S)^2$, which after being substituted into \eqref{21} yields the Riccati inequality
	\begin{align}\label{23}
	\operatorname{tr} S'(s) + [\operatorname{tr}S(s)]^2 + \operatorname{tr}R(s) \geq 0. 
	\end{align}
	We then appeal to the Ricci upper bound $\operatorname{Ric}_g \leq  b^2$ to infer from \eqref{23} the inequality 
	\begin{align}\label{23'}
		\operatorname{tr}S'(s) + [\operatorname{tr}S(s)]^2 +  b^2\operatorname{Id} \geq 0.
	\end{align}
	As explained in the proof of \cite[Theorem 16]{CGH20}, the desired estimate \eqref{22} then follows from \eqref{23'} and the argument in \cite[Corollary 1.6.2]{Kar89}: for $0<\delta<\delta_{r, b ,\mathcal{H}}$, one has $H_{\Sigma_\delta} \leq -\varphi(\delta)$ where $\varphi$ is the solution to the one-dimensional Riccati equation corresponding to the LHS of \eqref{23'}:
	\begin{align*}
	\begin{cases}
	\varphi'(t) + \varphi(t)^2 +  b^2= 0 \\
	\varphi(0) = -\mathcal{H}.
	\end{cases}
	\end{align*} 
	Since the solution $\varphi$ is given by
	\begin{align*}
	\varphi(t) = \begin{cases}
	-\frac{ b^2\tan( b  t) +  b  \mathcal{H} }{ b  - \mathcal{H} \tan( b  t)} & \text{if } b >0 \\[5pt]
	-\frac{\mathcal{H}}{1-\mathcal{H} t} & \text{if } b =0,
	\end{cases}
	\end{align*}
	this completes the proof of the theorem.
\end{proof}

Using Theorem \ref{D}, we now prove a counterpart to Lemma \ref{L1}:

\begin{lem}\label{L3}
	Let $(M^{n+1},g)$ be a smooth compact manifold with non-empty boundary $\Sigma^n$, and let $R>0$ denote the rolling radius of $\Sigma$ in $M$. Suppose:
	\begin{enumerate}
		\item $0<\kappa\leq \kappa_i$ and $H\leq \mathcal{H}$ on $\Sigma$ 
		\item $\operatorname{Ric}_g \leq  b^2$ and $\operatorname{Sec}_g \geq -\alpha^2$ in $M_r$ for some $ b \geq 0$, $0\leq\alpha\leq\kappa$ and $r\leq R$. 
	\end{enumerate}
	Then for each $0<\delta<\delta_{r, b ,\mathcal{H}}$, the mean curvature of $\Sigma_\delta$ satisfies the upper bound
	\begin{align}\label{27}
	H_{\Sigma_\delta} \leq P_{ b , \mathcal{H}}(\delta) - \delta^{-1}.
	\end{align}
\end{lem}

\begin{proof}
	Consider the ODE
	\begin{align*}
	\begin{cases}
	\psi'(t) + \psi(t)^2 - \alpha^2 = 0 \\
	\psi(0) = -\kappa.
	\end{cases}
	\end{align*}
	A direct computation yields the solution 
	\begin{align*}
	\psi(t) = \begin{cases}
	-\frac{\kappa}{1-\kappa t} & \text{if }\alpha = 0 \\
	-\frac{\kappa \alpha - \alpha^2 \tanh(\alpha t)}{\alpha - \kappa\tanh(\alpha t)} & \text{if }0<\alpha<\kappa \\
	-\kappa & \text{if }0<\alpha=\kappa, 
	\end{cases} 
	\end{align*}
	which exists on the maximal interval $t\in[0,T_{\kappa,\alpha})$ where
	\begin{align*}
	T_{\kappa,\alpha} = \begin{cases}
	\frac{1}{\kappa} & \text{if }\alpha=0 \\
	\alpha^{-1}\operatorname{arctanh}(\kappa^{-1}\alpha) & \text{if }0<\alpha<\kappa \\
	+\infty & \text{if }0<\alpha=\kappa.
	\end{cases}
	\end{align*} 
	By the principal curvature comparison theorem \cite[Theorem 10]{CGH20}, the rolling radius of $\Sigma$ in $M$ satisfies $R \leq T_{\kappa,\alpha}$, and for each $\delta\in(0,R)$, the principal curvatures $\kappa_i(\delta)$ of the parallel hypersurface $\Sigma_\delta$ satisfy the lower bound $\kappa_i(\delta) \geq -\psi(\delta)$.  In particular, the principal curvatures remain positive for $\delta\in(0,R)$, and thus the parallel hypersurfaces remain strictly convex (and in particular, 2-convex) for $\delta\in(0,R)$. Thus Theorem \ref{D} applies, yielding for each $0<\delta< \delta_{r, b ,\mathcal{H}}$ the upper bound \eqref{27}.
\end{proof}

We now complete the proof of Theorem \ref{C}:

\begin{proof}[Proof of Theorem \ref{C}]
	Following the proof of Lemma \ref{L2} verbatim, one sees that under the hypotheses of Lemma \ref{L3}, for each $\ep>0$, $0<\delta<\delta_{r, b ,\mathcal{H}}$ and $u\in C^\infty(M)$, it holds that
	\begin{align}\label{28}
	\ep\int_{\Sigma}|\nabla^T u|^2\,dS_g \leq \int_{M_\delta}|\nabla^2 u|^2\,dv_g + \ep \big(\ep+P_{ b ,\mathcal{H}}(\delta)\big)\int_{M_\delta}|\nabla u|^2\,dv_g.
	\end{align}
	One may then repeat the proof of Theorem \ref{A}, using \eqref{28} in place of \eqref{11} and replacing all appearances of $E$ by $P$. 
\end{proof}

\subsection{A spectral gap result when $\operatorname{Ric}_g<0$}\label{103}

We have seen that one may obtain a positive lower bound for $\sigma_1$ under relaxed convexity assumptions if the Ricci curvature of $M$ is positive near $\Sigma$ rather than just non-negative. A natural question to ask is whether our methods allow one to obtain a positive lower bound for $\sigma_1$ even if the Ricci curvature only has a negative lower bound, say if the boundary is sufficiently convex. Whilst we have not been able to answer this question affirmatively, we finish with the following observation, which provides a spectral gap result in this geometric setting as a direct consequence of Reilly's formula:
\begin{prop}
	Let $(M^{n+1},g)$ be a smooth compact manifold with non-empty boundary $\Sigma^n$, and suppose that:
	\begin{enumerate}
		\item $\operatorname{Ric}_g \geq -a^2$ in $M$ for some $a>0$, 
		\item $0<\sqrt{\frac{2a^2}{n}}<\kappa\leq \kappa_i$ on $\Sigma$.
	\end{enumerate}
Then the Steklov spectrum of $M$ admits the following spectral gap: 
\begin{align*}
	\{\sigma_1,\sigma_2,\dots\} \cap\bigg(\frac{a^2}{n\kappa}\,,\, \frac{\kappa}{2}\bigg) = \emptyset. 
\end{align*}
\end{prop}

\begin{proof}
	By Lemma \ref{L4}, we see that if $f$ is an eigenfunction associated to a Steklov eigenvalue $\sigma_j$ then, under the geometric assumptions of the proposition, 
	\begin{align*}
	0 \geq \int_M |\nabla^2 f|^2\,dv_g + (n\kappa\sigma_j - a^2)\int_M |\nabla f|^2\,dv_g + (\kappa-2\sigma_j)\int_{\Sigma}|\nabla^T f|^2\,dS_g.
	\end{align*}
	We therefore see that if $\sigma_j \geq \frac{a^2}{n\kappa}$, then $\sigma_j \geq \frac{\kappa}{2}$. This implication is non-trivial precisely when $\kappa>\sqrt{\frac{2a^2}{n}}$, which we have assumed to be the case. Likewise, if $\sigma_j \leq \frac{\kappa}{2}$, then it must be the case that $\sigma_j \leq \frac{a^2}{n\kappa}$, which again is a non-trivial implication precisely when $\kappa>\sqrt{\frac{2a^2}{n}}$. 
\end{proof}

\small
\bibliography{references}{}
\bibliographystyle{siam}
~\\
~\\
\indent \textsc{Johns Hopkins University, 404 Krieger Hall, Department of Mathematics, 3400 N. Charles Street, Baltimore, MD 21218, US.} \newline
\indent \textit{Email address:} \texttt{jdunca33@jhu.edu} 

~\\
\indent \textsc{Johns Hopkins University, 404 Krieger Hall, Department of Mathematics, 3400 N. Charles Street, Baltimore, MD 21218, US.} \newline
\indent \textit{Email address:} \texttt{akumar65@jhu.edu}

	\end{document}